\input louC.sty

\documentclass[a4paper,draft]{amsproc}
\usepackage{amsrefs}
\usepackage{amsmath}
\usepackage{mathrsfs}
\usepackage{amssymb}
\usepackage{amscd} 

\theoremstyle{plain}
 \newtheorem{thm}{\textbf{Theorem}}[section]
 \newtheorem{prop}{\textbf{Proposition}}[section]
 \newtheorem{lem}{\textbf{Lemma}}[section]
 \newtheorem{cor}{\textbf{Corollary}}[section]
\theoremstyle{definition}

\theoremstyle{remark}
 \newtheorem{rem}{\textbf{Remark}}[section]
 \numberwithin{equation}{section}

\renewcommand{\leq}{\leqslant}

\title{Product of truncated Hankel and truncated Toeplitz operators}

\subjclass[2010]{Primary 47}

\author[Chu]{\bfseries Cheng Chu}

\address{
Department of Mathematics and Statistics \\ 
Laval University  \\ 
Quebec City, QC \\
Canada}
\email{chengchu813@gmail.com}

\begin{document}

\vspace{18mm}
\setcounter{page}{1}
\thispagestyle{empty}

\begin{abstract}
A truncated Toeplitz operator is the compression of a classical Toeplitz operator on the Hardy space to a model space. A truncated Hankel operator is the compression of a Hankel operator on the Hardy space to the orthogonal complement of a model space. We study the product of a truncated Hankel operator and a truncated Toeplitz operator, and characterize when such a product is zero or compact.
\end{abstract}

\maketitle

\section{Introduction}
Let $\DD$ be the open unit disk in the complex plane. Let $L^2$ denote the Lebesgue space of square integrable functions on the unit circle $\partial\DD$. Let $L^\infty$ denote the space of essentially bounded Lebesgue measurable functions on $\partial\DD$. The Hardy space $H^2$ is the subspace of analytic functions on $\DD$ whose Taylor coefficients are square summable. Then it can also be identified with the subspace of $L^2$ of functions whose negative Fourier coefficients vanish. Let $H^\infty$ be the space of all bounded analytic functions on $\DD$.

Let $P$ be the orthogonal projections from $L^2$ to $H^2$. For $f\in L^2$, the Toeplitz operator $T_f$ and Hankel operator $H_f$ are defined by $$T_fh=P(fh),$$ and  $$H_f h=(I-P)(fh)$$ on the dense subset $H^2\cap L^\infty$ of $H^2$.
It is well-known that $T_f$ is bounded if and only if $f\in L^\infty$, and $H_f$ is bounded if and only if $(I-P)f$ is in the space $BMO$ of functions of bounded mean oscillation (see e.g. \cite{zhu2}).

In the past several years, there has been a vigorous development in the study of truncated Toeplitz operators since Sarason's seminal paper \cite{sar07}.
An analytic function $\Gt$ is called an inner function if $|\Gt|=1$ a.e. on $\partial\DD$. For each non-constant inner function $\Gt$, the so-called model space is $$K_\Gt=H^2\ominus \Gt H^2.$$ It is a reproducing kernel Hilbert space with reproducing kernels $$k_w^{\Gt}(z)=\frac{1-\overline{\Gt(w)}\Gt(z)}{1-\bar{w}z}.$$

Let $P_\Gt$ denote the orthogonal projection from $L^2$ onto $K_\Gt$,
\beq\label{Pt}
P_{\Gt}f=P f-\Gt P(\bar{\Gt}f).
\eeq
For $\varphi\in L^2$, the truncated Toeplitz operator $A_\Gvp^\Gt$ and the truncated Hankel operator $H_\Gvp^\Gt$ are defined by $$A^{\Gt}_\varphi f=P_\Gt (\varphi f)$$ and $$H^\Gt_\Gvp f=(I-P_\Gt) (\Gvp f),$$ on the dense subset $K_\Gt \cap H^{\infty}$ of $K_\Gt$. In particular, $K_\Gt \cap H^{\infty}$ contains all reproducing kernels $k_w^{\Gt}$.  The symbol $\Gvp$ is never unique for $A_\Gvp^\Gt$ and it is proved in \cite{sar07} that
\begin{thm}\label{T0}
The operator $A^{\Gt}_\varphi=0$ if and only if
$$
\Gvp\in \Gt H^2+\ol{\Gt H^2}.
$$
\end{thm}
For truncated Hankel operators, we have \cite{mazheng2}
\begin{thm}\label{H0}
The operator $H^\Gt_\Gvp=0$ if and only if $\Gvp$ is a constant.
\end{thm}

The compactness of Toeplitz and Hankel operator are known. The only compact Toeplitz operator is the zero operator (see e.g. \cite{dou72}, \cite{zhu}). For Hankel operators, Hartman's Criterion (see e.g. \cite{pel02}, \cite{zhu}) asserts that
$H_f$ is compact if and only if $f\in H^\infty+C,$ where $C$ denotes the space of continuous functions on the unit circle.
The problem of characterizing the compactness for product (or sum of products) of Hankel and Toeplitz operators turns out to be much more difficult. The abstract theory of maximal ideal space plays an important role in those problems \cite{ax78}, \cite{vol82}, \cite{zh96}, \cite{guz98}, \cite{gor99}, \cite{guo03}, \cite{guo05}, \cite{chu15}.

Assuming boundedness of the symbol, compact truncated Toeplitz operators and truncated Hankel operators are characterized in \cite{mazheng},
 \cite{mazheng2}. We will explain the notations below in Section 2.
\begin{thm}\label{ctto}\cite{mazheng}*{Theorem 1}
Let $f\in L^\infty$. Then $A^\Gt_f$ is compact if and only if for every $m\in M(H^\infty+C)$, either $$\Gt|_{S_m} \m{is constant} ,$$ or $$f|_{S_m}\in \Gt|_{S_m}H^2(m)+\bar{\Gt}|_{S_m}\ol{H^2(m)}.$$
\end{thm}
\begin{thm}\cite{mazheng2}*{Theorem 1.3}
Let $f\in L^\infty$. Then $H^\Gt_f$ is compact if and only if for every $m\in M(H^\infty+C)$, either $\Gt|_{S_m}$ or $f|_{S_m}$ is constant.
\end{thm}
A natural question is to study the product of a truncated Hankel operator and a truncated Toeplitz operator $H^\Gt_f A^\Gt_g$ on the model space. In this paper, we characterize when $H^\Gt_f A^\Gt_g$ is zero (Theorem \ref{zero}) and when $H^\Gt_f A^\Gt_g$ is compact (Theorem \ref{compact}).

\section{Preliminaries}

For $\Gvp\in L^2$, let $S_\Gvp: [H^2]^\perp\to [H^2]^\perp$ be the operator $S_\Gvp h=(I-P)(\Gvp h)$.
We will frequently use the following basic properties of Toeplitz and Hankel operators on $H^2$ (see e.g. \cite{chu15}).
\begin{prop}\label{prop}
Let $f\in L^\infty, g\in L^2$.
\begin{enumerate}
\item $T_{fg}=T_fT_g+H^*_{\barf} H_g$.
\item $H_{fg}=H_fT_g+S_fH_g$.
\item If $g\in H^2$, then $T_fT_g=T_{fg}$ and $H_f T_g=H_{fg}$.
\item If $f\in H^2$, then $H_{fg}=S_fH_g$.
\end{enumerate}
\end{prop}

Define an antiunitary operator $V$ on $L^2$ by: $$(Vf)(z)=\bz\ol{f(z)}.$$
It is easy to check that
\beq\label{V}
V^{-1}H_\Gvp V=H^*_\Gvp.
\eeq

Define a unitary operator $U$ on $L^2$ by: $$(Uf)(z)=\bz \tif(z),$$ where $\tif(z)=f(\bz)$.
For each $z\in \DD$, let $k_z$ denote the normalized reproducing kernel at z: $$k_z(w)=\frac{\sqrt{1-|z|^2}}{1-\bz w},$$
and $\phi_z$ be the M\"obius transform: $$\phi_z(w)=\frac{z-w}{1-\bz w}.$$

Let $X, Y$ be Hilbert spaces. Let $x\in X, y\in Y$. Define $x\otimes y$ to be the following rank one operator from $Y$ to $X$:
$$
(x\otimes y)(f)=\langle f,y\rangle_Y x.
$$

The operator $T_{\phi_z}T_{\ol{\phi}_{z}}$ is the orthogonal projection onto $H^2\ominus \{k_z\}$, thus
$$
I-T_{\phi_z}T_{\ol{\phi}_{z}}=k_z\otimes k_z.
$$
It is easy to check that
$$
US_fU=T_{\tif}, \q S^*_f=S_{\barf}.
$$
Therefore we have the following identity:
\beq\label{S}
S^*_{\phi_z}S_{\phi_z}=1-(Uk_{\bz})\otimes (Uk_{\bz}).
\eeq

By Theorem \ref{T0}, we may assume $A^\Gt_g$ has a symbol in $K_\Gt+\ol{K_\Gt}$.
The following lemma shows that a product of a truncated Hankel operator and a truncated Toeplitz operator can be written as a sum of two operators whose ranges are orthogonal.
\begin{lem}\label{eq1}
Let $f\in L^\infty$ and $g\in K_\Gt+\ol{K_\Gt}$.
$$
H^\Gt_f A^\Gt_g H^*_{\bar{\Gt}}=(H_{\bar{\Gt}} T_{\bg} H^*_f-H_{\Gt f} H^*_{\bg})-T_\Gt(T_{\bar{\Gt} f} T_g H^*_{\bar{\Gt}}-T_f H^*_{\bg}).
$$
\end{lem}
\begin{proof}
Let $u\in K_\Gt$. By \eqref{Pt}, we have
$$
A^\Gt_{g}u=P(gu)-\Gt P(\bar{\Gt}gu)= T_g u-T_\Gt T_{\bar{\Gt}g} u,
$$
$$
H^\Gt_f u=fu-A^\Gt_f u=fu-(T_f u-T_\Gt T_{\bar{\Gt}f}u)=H_f u-T_\Gt T_{\bar{\Gt}f}u.
$$
Thus
\begin{align*}
H^\Gt_f A^\Gt_g u=&H_f(T_g u-T_\Gt T_{\bar{\Gt}g} u)-\Gt T_{\bar{\Gt}f}(T_g u-T_\Gt T_{\bar{\Gt}g} u)\\
=&H_f T_g u -H_{\Gt f}T_{\bar{\Gt}g} u-\Gt(T_{\bar{\Gt}f}T_g u-T_f T_{\bar{\Gt}g}u).
\end{align*}
Notice that $H^*_{\bar{\Gt}}: [H^2]^{\perp}\to K_\Gt$, thus
\begin{align*}
H^\Gt_f A^\Gt_gH^*_{\bar{\Gt}} =&H_f T_g H^*_{\bar{\Gt}} -H_{\Gt f}T_{\bar{\Gt}g} H^*_{\bar{\Gt}}-T_\Gt(T_{\bar{\Gt}f}T_g H^*_{\bar{\Gt}}-T_f T_{\bar{\Gt}g}H^*_{\bar{\Gt}})\\
=&(H_{\bar{\Gt}} T_{\bg} H^*_f-H_{\Gt f} H^*_{\bg})-T_\Gt(T_{\bar{\Gt} f} T_g H^*_{\bar{\Gt}}-T_f H^*_{\bg}).
\end{align*}
The last equality follows from $\bar{\Gt}g\in \ol{H^2}$ and Proposition \ref{prop}.
\end{proof}
With respect to the decomposition $[H^2]^\perp=\bar{\Gt}K_{\Gt}\oplus\bar{\Gt}[H^2]^\perp$, the operator $H^*_{\Gt}$ maps $\bar{\Gt}K_{\Gt}$ onto $K_\Gt$, and is zero on $\bar{\Gt}[H^2]^\perp$.
Therefore Lemma \ref{eq1} implies that $H^\Gt_f A^\Gt_g$ is zero or compact if and only if $H_{\bar{\Gt}} T_{\bg} H^*_f-H_{\Gt f} H^*_{\bg}$ and $T_{\bar{\Gt} f} T_g H^*_{\bar{\Gt}}-T_f H^*_{\bg}$ are zero or compact, respectively.

The next lemma suggests that we need to study the property of a sum of operators of the forms $H^*_fH_g$ and $H_fT_g$.
\begin{lem}\label{eq2}
Let $f\in L^\infty$, $g=g_1+\bg_2$, where $g_1, g_2\in K_\Gt$. Then
$$
V^{-1}(H_{\bar{\Gt}} T_{\bg} H^*_f-H_{\bg} H^*_{\Gt f})V=H^*_{\bar{\Gt}}H_{fg_1}+H^*_{\bar{\Gt}g_2}H_f-H^*_{\bg_1}H_{\Gt f}.
$$
$$
(T_{\bar{\Gt} f} T_g H^*_{\bar{\Gt}}-T_f H^*_{\bg})^*=H_{\bar{\Gt}}T_{\Gt\ol{fg_1}}+H_{\bar{\Gt}g_2}T_{\Gt\barf}-H_{\bg_1}T_{\barf}.
$$
\end{lem}
\begin{proof}
By Proposition \ref{prop}, we have
\begin{align*}
H_{\bar{\Gt}} T_{\bg} H^*_f-H_{\bg} H^*_{\Gt f}=&H_{\bar{\Gt}} T_{\bg_1} H^*_f+H_{\bar{\Gt}} T_{g_2} H^*_f -H_{\bg_1} H^*_{\Gt f}\\
=&H_{\bar{\Gt}} H^*_{fg_1}+H_{\bar{\Gt}g_2} H^*_f -H_{\bg_1} H^*_{\Gt f}.
\end{align*}
Using \eqref{V}, we have
\begin{align*}
&V^{-1}(H_{\bar{\Gt}} T_{\bg} H^*_f-H_{\bg} H^*_{\Gt f})V\\
=&V^{-1}H_{\bar{\Gt}}VV^{-1} H^*_{fg_1}V+V^{-1}H_{\bar{\Gt}g_2}VV^{-1} H^*_fV -V^{-1}H_{\bg_1}VV^{-1} H^*_{\Gt f}V\\
=&H^*_{\bar{\Gt}}H_{fg_1}+H^*_{\bar{\Gt}g_2}H_f-H^*_{\bg_1}H_{\Gt f}.
\end{align*}
And by Proposition \ref{prop},
\begin{align*}
T_{\bar{\Gt} f} T_g H^*_{\bar{\Gt}}-T_f H^*_{\bg}=&T_{\bar{\Gt} f} T_{g_1} H^*_{\bar{\Gt}}+T_{\bar{\Gt} f} T_{\bg_2} H^*_{\bar{\Gt}}-T_f H^*_{\bg_1}\\
=&T_{\bar{\Gt} fg_1} H^*_{\bar{\Gt}}+T_{\bar{\Gt} f} H^*_{\bar{\Gt}g_2}-T_f H^*_{\bg_1}.
\end{align*}
Taking adjoint on both sides, we get the second equation.
\end{proof}
To state the results on compactness, we need some notations for the maximal ideal space.
For a uniform algebra $B$, let $M(B)$ denote the maximal ideal space of $B$, the space of nonzero multiplicative linear functionals of $B$. Given the weak-star topology of $B^*$, which is called the Gelfand topology, $M(B)$ is a compact Hausdorff space. Identify every element in $B$ with its Gelfand transform, we view $B$ as a uniformly closed algebra of continuous functions on $M(B)$. See \cite{gar81}*{Chapter V} for further discussions of uniform algebra.

For each $\Gz\in \DD$, there exist $m_\Gz \in M(H^\infty)$ such that $m_\Gz(z)=\Gz$, where $z$ denotes the coordinate function. It is well known that $\Gz\rightarrow m_\Gz$ is a homeomorphic embedding from $\DD$ into $M(H^\infty)$, thus we identify $\DD$ as a subset of $M(H^\infty)$. By Carleson's Corona Theorem \cite{car62}, $\DD$ is dense in $M(H^\infty)$. Moreover, $M(H^\infty+C)=M(H^\infty)\backslash \DD$ \cite{sar67}.

For any $m$ in $M(H^\infty)$, there exists a representing measure $\mu_m$ on $M(L^\infty)$ such that $m(f)=\int_{M(L^\infty)} fd\mu_m$, for all $f\in L^\infty$ (see e.g. \cite{gar81}*{p. 193}). Let $S_m$ be the support of $\mu_m$. For subspace $\cM$ of $L^2$, $\cM_m$ denotes $\cM|_{S_m}$. For a function $\Gvp\in L^\infty$, let $[\Gvp]_m$ denote the element in $L^\infty_m /H^\infty_m$ which contains $\Gvp$. We say
$$\lim_{z\to m}\Gvp(z)=0,$$ if for every net $\{z_{\Ga}\}\subset\DD$ converging to $m$, $$\lim_{z_{\Ga}\to m}\Gvp(z_{\Ga})=0.$$

\begin{thm}\cite{guo05}
Let $T$ be a finite sum of finite products of Toeplitz operators. Then $T$ is a compact perturbation of a Toeplitz operator if and only if
\beq\label{guo}
\lim_{|z|\to 1^-}||T-T^*_{\phi_z}TT_{\phi_z}||=0.
\eeq
\end{thm}
By the Corona Theorem, \eqref{guo} can be restated as: for every $m\in M(H^\infty+C)$,
$$
\lim_{z\to m}||T-T^*_{\phi_z}TT_{\phi_z}||=0.
$$
The symbol map $\Gs$ that sends every Toepltiz operator $T_\phi$ to its symbol $\phi$ was introduced in \cite{dou72} and can be defined on the Toeplitz algebra, the closed algebra generated by Toeplitz operators. Barr\'ia and Halmos in \cite{ba82} showed that $\Gs$ can be extended to a $*$-homomorphism on the Hankel algebra, the closed algebra generated by Toeplitz and Hankel operators. And they also showed that the symbols of compact operators and Hankel operators are zero. Therefore we have the following
\begin{cor}\label{gz}
Let $T$ be a finite sum of finite products of Toeplitz operators and $\Gs(T)=0$. Then $T$ is a compact if and only if for every $m\in M(H^\infty+C)$,
\beq\label{guo2}
\lim_{z\to m}||T-T^*_{\phi_z}TT_{\phi_z}||=0.
\eeq
\end{cor}

We will use the following identities to verify \eqref{guo2}.
\begin{lem}\label{Tensor}\cite{zh96}, \cite{guo03}
Let $f, g\in L^2$.
$$
H^*_fH_g-T^*_{\phi_z}H^*_fH_gT_{\phi_z}=V[(H_fk_z)\otimes (H_gk_z)]V^*.
$$
$$
S_{\phi_z}H_fT_gT_{\bar{\phi_z}}-H_fT_g=-(H_fT_gk_z)\otimes k_z+[(H_fk_{z})\otimes(H^*_{g}Uk_{\bz})]T_{\bar{\phi_z}}.
$$
\end{lem}
\begin{rem}
In \cite{zh96}, \cite{guo03}, Hankel operators are defined in an alternative way as an operator from $H^2$ to $H^2$ as: $$\cH _fh=PU(fh).$$
It is easy to verify that $\cH_f=U H_f$.
\end{rem}

The next results interpret the local condition on the support set in an elementary way.
\begin{lem}\label{lim}\cite{gor99}*{Lemma 2.5, 2.6}
Let $f\in L^\infty$, $m\in M(H^\infty+C)$. Then the following are equivalent:
\begin{enumerate}
\item $f|_{S_m}\in H^\infty|_{S_m}.$
\item $\lim\limits_{z\to m}||H_f k_z||=0.$
\item $\lim\limits_{z\to m}||H^*_f U k_{\bz}||=0.$
\end{enumerate}
\end{lem}

\begin{lem}\label{*lim} \cite{guo03}*{Lemma 17,18}
Let $f,g\in L^\infty$, $m\in M(H^\infty+C)$.
\begin{enumerate}
\item If
$$\lim_{z\to m}||H_f k_z||=0,$$ then $$\lim_{z\to m}||H_fT_g k_z||=0.$$
\item If
$$\lim_{z\to m}||H^*_f U k_{\bz}||=0,$$ then $$\lim_{z\to m}||H^*_f S_g Uk_{\bz}||=0.$$
\end{enumerate}
\end{lem}

\section{Zero product}

In this section, we characterize zero product of a truncated Hankel operator and a truncated Toeplitz operator.

\begin{lem}\label{q0}
Let $f_i, g_i\in L^2$, $i=1,...,n$. Let $\pi: L^2\to L^2/H^2$ be the quotient map.
If $\{\pi(f_i) \}_{i=1}^n$ are linearly independent, and either
\beq\label{q01}
\sum_{i=1}^n (H_{f_i}1)\otimes (H_{g_i}1)=0,
\eeq
or
\beq\label{q02}
\sum_{i=1}^n (H_{f_i}1)\otimes (H^*_{g_i}\bz)=0,
\eeq
then
$g_i\in H^2$, for $i=1,...,n$.
\end{lem}
\begin{proof}
First notice that for $f\in L^2$ $$H_f 1=0\iff H^*_f\bz=0\iff f\in H^2.$$
Suppose \eqref{q01} holds and the conclusion is not true. A similar argument can be applied to \eqref{q02}. We may assume $g_1\notin H^2$, then $||H_{g_1}1||>0$.
Apply the operator $\sum_{i=1}^n (H_{f_i}1)\otimes (H_{g_i}1)$ to $H_{g_1}1$, we have
$$
\sum_{i=1}^n \la H_{g_1}1, H_{g_i}1\ra H_{f_i}1=0.
$$
Then
$$
H_{\sum_{i=1}^n a_if_i}1=0,
$$
where $a_i= \la H_{g_1}1, H_{g_i}1\ra$ and $a_1\neq 0$. Therefore $\sum_{i=1}^n a_if_i\in H^2$, which is a contradiction.
\end{proof}

Let $K_\Gt+\CC \Gt$ denote the set $$K_\Gt+\CC \Gt=\{f=f_1+c\Gt|f_1\in K_\Gt, c\in\CC \}=\{f\in H^2| \Gt\barf\in H^2\}.$$
By Theorem \ref{T0}, $A^\Gt_g$ has a symbol in $g_1+\bg_2$, where $g_1, g_2\in K_\Gt$. Furthermore, $g_1$ and $g_2$ are uniquely determined if we fix the value of one of them at the origin \cite{bcfmd}. Therefore we may assume $g_1(0)=0$ and we characterize the zero product of a truncated Hankel operator and a truncated Toeplitz operator.

\begin{thm}\label{zero}
Let $f\in L^2$. Let $g=g_1+\ol{g_2}$, where $g_1, g_2 \in K_\Gt$ and $g_1(0)=0$. Then $$H^\Gt_f A^\Gt_g=0$$ if and only if one of the following holds
\begin{enumerate}
\item $g=0$.
\item $f$ is a constant.
\item $f\in K_\Gt+\CC \Gt$, and $\bg, \ol{fg}\in H^2$.
\item there exist $\Ga, \Gb\in \CC$ such that
\begin{enumerate}
\item $g=g_1+\Ga \bar{\Gt}g_1+\Gb(1-\Gt(0)\bar{\Gt}).$
\item $f(g_1+\Gb), f(\Ga-\Gt)\in K_\Gt+\CC \Gt$.
\item $\bar{\Gb}\barf+\bar{\Ga}\Gt\bg_1 \in H^2$.
\end{enumerate}
\end{enumerate}
\end{thm}
\begin{proof}
By Lemma \ref{eq1}, Lemma \ref{eq2}, $H^\Gt_f A^\Gt_g=0$ if and only if
\beq\label{p01}H^*_{\bar{\Gt}}H_{fg_1}+H^*_{\bar{\Gt}g_2}H_f-H^*_{\bg_1}H_{\Gt f}=0,\eeq
and
\beq\label{p02}H_{\bar{\Gt}}T_{\Gt\ol{fg_1}}+H_{\bar{\Gt}g_2}T_{\Gt\barf}-H_{\bg_1}T_{\barf}=0.\eeq
Necessity: Suppose $H^\Gt_f A^\Gt_g=0$. By Lemma \ref{Tensor} and let $z=0$, we have
\beq\label{01}
H_{\bar{\Gt}}1\otimes H_{fg_1}1+H_{\bar{\Gt}g_2}1\otimes H_f1+H_{\bg_1}1\otimes H_{-\Gt f}1=0,
\eeq
and
\beq\label{02}
H_{\bar{\Gt}}1\otimes H^*_{\Gt\ol{fg_1}}\bz+H_{\bar{\Gt}g_2}1\otimes H^*_{\Gt\barf}\bz+H_{\bg_1}1\otimes H^*_{-\barf}\bz=0.
\eeq
If $\{\pi(\bar{\Gt}), \pi(\bar{\Gt}g_2), \pi(\bg_1)\}$ are linearly independent, then by Lemma \ref{q0} we must have $f\in H^2$ and $\barf\in H^2$. Thus condition (2) holds.
Now we assume there exist $t_1, t_2, t_3\in\CC$, not all $0$, such that
$$
t_1\bar{\Gt}+t_2\bar{\Gt}g_2+t_3\bg_1\in H^2,
$$
which means
\beq\label{C}
\bt_1\Gt+\bt_2\Gt\bg_2+\bt_3g_1=C\in \CC.
\eeq
Since $(\Gt\bg_2)(0)=g_1(0)=0$, we see that
\beq\label{t1}
\bt_1\Gt(0)=C.
\eeq
On the other hand,
\beq\label{t11}
C\ol{\Gt(0)}=\la C, \Gt\ra=\la\bt_1\Gt+\bt_2\Gt\bg_2+\bt_3g_1, \Gt  \ra=\bt_1+\bt_2\ol{g_2(0)}.
\eeq
Combing \eqref{t1} and \eqref{t11}, we have $$\bt_2\ol{g_2(0)}=\bt_1(|\Gt(0)|^2-1).$$
We consider the following two case:

Case I: If $t_2=0$, then $t_1=C=0$, and thus $g_1=0$. We have $$H^*_{\bar{\Gt}g_2}H_f=H_{\bar{\Gt}g_2}T_{\Gt\barf}=0,\q H_{\bar{\Gt}g_2}1\otimes H_f1=H_{\bar{\Gt}g_2}1\otimes H^*_{\Gt\barf}\bz=0.$$
Assume $g_2\neq 0$, then $\bar{\Gt}g_2\notin H^2$. Then by Lemma \ref{q0},
$f, \Gt\barf\in H^2$. Also $H_{\bar{\Gt}g_2}T_{\Gt\barf}=H_{\barf g_2}=0$ implies $\barf g_2=\ol{fg}\in H^2.$ Therefore we get condition (3).

Case II: If $t_2\neq 0$, we may assume $t_2=1$, then $$t_1=\frac{\ol{g_2(0)}}{|\Gt(0)|^2-1}, \q C=\frac{\Gt(0)g_2(0)}{|\Gt(0)|^2-1}.$$
Denote $\Gb=-t_1$. We may restate \eqref{C} as: there exists $\Ga\in\CC$ such that
$$\Gt\bg_2=\Gb\Gt+\Ga g_1-\Gt(0)\Gb,$$
or
$$\bg_2=\Gb(1-\Gt(0)\bar{\Gt})+\Ga \bar{\Gt}g_1,$$
which gives condition (4a).
By \eqref{01}, \eqref{02} we have
\begin{align*}
&H_{\bar{\Gt}}1\otimes H_{fg_1}1+H_{\bar{\Gt}g_2}1\otimes H_f1+H_{\bg_1}1\otimes H_{-\Gt f}1\\
=&H_{\bar{\Gt}g_2-\bar{\Gb}\bar{\Gt}-\bar{\Ga}\bg_1}1\otimes H_f1+H_{\bar{\Gt}}1\otimes H_{fg_1+\Gb f}1+H_{\bg_1}1\otimes H_{-\Gt f+\Ga f}1
\end{align*}
and
\begin{align*}
&H_{\bar{\Gt}}1\otimes H^*_{\Gt\ol{fg_1}}\bz+H_{\bar{\Gt}g_2}1\otimes H^*_{\Gt\barf}\bz+H_{\bg_1}1\otimes H^*_{-\barf}\bz\\
=&H_{\bar{\Gt}g_2-\bar{\Gb}\bar{\Gt}-\bar{\Ga}\bg_1}1\otimes H^*_{\Gt\barf}\bz+H_{\bar{\Gt}}1\otimes H^*_{\Gt\ol{fg_1}+\bar{\Gb}\Gt\barf}\bz+H_{\bg_1}1\otimes H^*_{-\barf+\bar{\Ga}\Gt\barf}\bz.
\end{align*}
Notice that in this case, $\pi(\bar{\Gt}), \pi(\bg_1)$ are linearly independent, and
\beq\label{00}H_{\bar{\Gt}g_2-\bar{\Gb}\bar{\Gt}-\bar{\Ga}\bg_1}=0.\eeq By Lemma \ref{q0}, the functions
$$
f(g_1+\Gb), \,f(\Ga-\Gt),\, \Gt\barf (\bg_1+\bar{\Gb}),\, \barf (\bar{\Ga}\Gt-1)
$$
are all analytic, which imply condition (4b). Condition (4c) follows from the identity
\begin{align}\label{ht}
0=&H_{\bar{\Gt}}T_{\Gt\ol{fg_1}}+H_{\bar{\Gt}g_2}T_{\Gt\barf}-H_{\bg_1}T_{\barf}\\
\nnb=&H_{\bar{\Gt}g_2-\bar{\Gb}\bar{\Gt}-\bar{\Ga}\bg_1}T_{\Gt\barf}+H_{\bar{\Gt}}T_{\Gt\ol{fg_1}+\bar{\Gb}\Gt\barf}+H_{\bg_1}T_{-\barf+\bar{\Ga}\Gt\barf}\\
\nnb=&H_{\ol{fg_1}+\bar{\Gb}\barf}+H_{-\ol{fg_1}+\bar{\Ga}\Gt\ol{fg_1}}=H_{\bar{\Gb}\barf+\bar{\Ga}\Gt\bg_1}.
\end{align}

Sufficiency: If condition (1) holds, it is obvious.

If condition (2) holds, it follows from Theorem \ref{H0}.

If condition (3) holds, then $H_{fg_1}=H_{\bg_1}=H_{\Gt f}=0$ , which give \eqref{p01}.
Also
$$
H_{\bar{\Gt}g_2}T_{\Gt\barf}=H_{\barf g_2}=H_{\ol{fg}}=0,
$$
and thus \eqref{p02} holds.

If condition (4) holds, we have \eqref{00}. Then \eqref{p02} follows from \eqref{ht}, and we can verify \eqref{p01} as
\begin{align*}
&H^*_{\bar{\Gt}}H_{fg_1}+H^*_{\bar{\Gt}g_2}H_f-H^*_{\bg_1}H_{\Gt f}\\
=&H^*_{\bar{\Gt}g_2-\bar{\Gb}\bar{\Gt}-\bar{\Ga}\bg_1}H_f+H^*_{\bar{\Gt}}H_{fg_1+\Gb f}+H^*_{\bg_1}H_{-\Gt f+\Ga f}=0.
\end{align*}

\end{proof}

\begin{cor}
Let $f\in L^2$. Let $g=g_1+\ol{g_2}+\Gt g_3+\ol{\Gt g_4}$, where $g_1, g_2 \in K_\Gt$ and $g_3, g_4\in H^2$. Then $$H^\Gt_f A^\Gt_g=0$$ if and only if one of the following holds
\begin{enumerate}
\item $g\in \Gt H^2+\ol{\Gt H^2}$.
\item $f$ is a constant.
\item $f\in K_\Gt+\CC \Gt$, $g_1=0$ and $\ol{f(g_1+\ol{g_2})}\in H^2$.
\item there exist $\Ga, \Gb\in \CC$ such that
\begin{enumerate}
\item $g=g_1+\Ga \bar{\Gt}g_1+\Gb(1-\Gt(0)\bar{\Gt})+\Gt g_3+\ol{\Gt g_4}.$
\item $f(g_1+\Gb), f(\Ga-\Gt)\in K_\Gt+\CC \Gt$.
\item $\bar{\Gb}\barf+\bar{\Ga}\Gt\bg_1 \in H^2$.
\end{enumerate}
\end{enumerate}
\end{cor}

\section{Compact product}

In this section, we characterize compact product of truncated Hankel and Toeplitz operators with bounded symbol.

\begin{thm}\label{CC}
Let $f_i, g_i\in L^\infty$, $i=1,...,n$. Let $$K_1=\sum_{i=1}^nH^*_{f_i}H_{g_i}, \q K_2=\sum_{i=1}^nH_{f_i}T_{g_i}.$$
Then \begin{enumerate}
\item $K_1$ is compact if and only if for every $m\in M(H^\infty+C)$,
$$
\lim_{z\to m}||\sum_{i=1}^n (H_{f_i}k_z)\otimes (H_{g_i}k_z)||=0.
$$
\item $K_2$ is compact if and only if for every $m\in M(H^\infty+C)$,
$$
\lim_{z\to m}||\sum_{i=1}^n (H_{f_i}k_z)\otimes (H^*_{g_i}U k_{\bz})||=0,
$$
and
$$
\lim_{z\to m}||K_2^*Uk_{\bz}||=0.
$$
\end{enumerate}
\end{thm}
\begin{proof}
(1) By Proposition \ref{prop}, $K_1$ is a finite sum of finite products of Toeplitz operators with $\Gs(K_1)=0$, therefore the conclusion follows from Lemma \ref{Tensor} and Corollary \ref{gz}.

(2) Suppose $K_2$ is compact. Since $k_z\to 0$ weakly and $Uk_{\bz}\to 0$ weakly, we have
$$
\lim_{z\to m} ||K_2k_z||=||K_2^*Uk_{\bz}||=0.
$$
Using a similar argument as in \cite{guo03}*{Lemma 9}, we have
$$
\lim_{z\to m}||K_2-S_{\phi_z}K_2T_{\ol{\phi_z}}||=0.
$$
It then follows directly from Lemma \ref{Tensor}.

For the sufficiency part, notice that $K_2^*K_2$ is a finite sum of finite products of Toeplitz operators with $\Gs(K_2)=0$. By Corollary \ref{gz}, we only need to show that for every $m\in M(H^\infty+C)$,
\beq\label{CC1}
\lim_{z\to m}||K_2^*K_2-T^*_{\phi_z}K_2^*K_2T_{\phi_z}||=0.
\eeq
By \cite{chu15}*{Lemma 5.2},
$$
K_2T_{\phi_z}=S_{\phi_z}K_2-\sum_{i=1}^n (H_{f_i}k_z)\otimes (H^*_{g_i}U k_{\bz}).
$$
Let $$F_z=\sum_{i=1}^n (H_{f_i}k_z)\otimes (H^*_{g_i}U k_{\bz}).$$
Then
\begin{align*} T^*_{\phi_z}K_2^*K_2T_{\phi_z}&=(K_2T_{\phi_z})^*(K_2T_{\phi_z})=(K_2^*S^*_{\phi_z}-F^*_z)(S_{\phi_z}K_2-F_z)\\
&=K^*_2S^*_{\phi_z}S_{\phi_z}K_2-K^*_2S^*_{\phi_z}F_z-F^*_zS_{\phi_z}K_2+F^*_zF_z\\
&=K_2^*K_2+(K_2^*Uk_{\bz})\otimes(K_2^*Uk_{\bz})-K^*_2S^*_{\phi_z}F_z-F^*_zS_{\phi_z}K_2+F^*_zF_z.
\end{align*}
The last equality follows from \eqref{S}. Thus \eqref{CC1} holds.
\end{proof}

For convenience, we introduce the following notations. For functions $f_i, i=1,...,n$. Let $\vec{F}=(f_1,...,f_n)^{T}$, $[\vec{F}]_m=([f_1]_m,...,[f_n]_m)^{T}$. Say $\vec{F}\in L^\infty$ if each $f_i\in L^\infty$, for all $i$. Say $\vec{F}|_{S_m}\in H^\infty|_{S_m}$ if $f_i|_{S_m}\in H^\infty|_{S_m}$, for all $i$. For $u\in H^2$, let
$$H_{\vec{F}}u=\sum_{i=1}^n H_{f_i}u, \q H^*_{\vec{F}}u=\sum_{i=1}^n H^*_{f_i}u,  \q ||H_{\vec{F}}u||=\sum_{i=1}^n ||H_{f_i}u||.$$
For $\vec{G}=(g_1,...,g_n)^T\in L^\infty$, $v\in H^2$, let
$$
H_{\vec{F}}u\otimes H_{\vec{G}}v=\sum_{i=1}^n H_{f_i}u\otimes H_{g_i}v.
$$

The next lemma (part (1) appeared in \cite{guz98}) is essential in the proof of the main theorem.
\begin{lem}\label{qc}
Let $\vec{F}=(f_1,..,f_n)^T, \vec{G}=(g_1,...,g_n)^T$ and $\vec{F},\vec{G}\in L^\infty$. Let $m\in M(H^\infty+C)$. Assume $\{[f_1]_m,...,[f_N]_m\}$ forms a basis for $\{[f_1]_m,...,[f_n]_m\}$, for some $N\leq n$. Then there exists a scalar matrix $B$ such that
$$
([f_1]_m,..,[f_n]_m)^T=B([f_1]_m,..,[f_N]_m)^T.
$$
Let $A=(B, 0)_{n\times n}$.
\begin{enumerate}
\item If
\beq\label{qc1}
\lim_{z\to m}||H_{\vec{F}} k_z\otimes H_{\vec{G}}k_z||=0,
\eeq
then $(A^*\vec{G})|_{S_m}\in H^\infty|_{S_m}$.
\item If
\beq\label{qc2}
\lim_{z\to m}||H_{\vec{F}} k_z\otimes H^*_{\vec{G}}U k_{\bz}||=0,
\eeq
then
$(A^T\vec{G})|_{S_m}\in H^\infty|_{S_m}$. Moreover, if in addition $$\sum_{i=1}^n H_{f_i}T_{g_i}$$ is compact, then
$(\vec{G}^T A \vec{F})|_{S_m}\in H^\infty|_{S_m}$.

\end{enumerate}
\end{lem}
\begin{proof}
(1) Suppose \eqref{qc1} holds and the conclusion is not true.
It is easy to see $A\vec{[F]_m}=\vec{[F]_m}$, then $((I-A)\vec{F})|_{S_m}\in H^\infty|_{S_m}$. Thus by Theorem \ref{lim}, $$\lim_{z\to m}||H_{(I-A)\vec{F}}k_z||=0.$$
On the other hand, we have
\beq\label{qcc1}
H_{\vec{F}} k_z\otimes H_{\vec{G}}k_z=H_{(I-A)\vec{F}} k_z\otimes H_{\vec{G}}k_z+H_{\vec{F}} k_z\otimes H_{A^*\vec{G}}k_z.\eeq
Thus
$$
\lim_{z\to m}||H_{\vec{F}} k_z\otimes H_{A^*\vec{G}}k_z||=0.
$$
We need to show $$\lim_{z\to m}||H_{A^*\vec{G}}k_z||=0.$$
Suppose it is not true. Let $A^*\vec{G}=(\tig_1,...,\tig_N, 0,...,0)^T$, we may assume $$\lim_{z\to m}||H_{\tig_1}k_z||>0.$$
Apply the operator $H_{\vec{F}} k_z\otimes H_{A^*\vec{G}}k_z$ to $H_{\tig_1}k_z$, we have
$$
\lim_{z\to m}||\sum_{i=1}^N \la H_{\tig_1}k_z, H_{\tig_i}k_z\ra H_{f_i}k_z||=0.
$$
Since $\frac{\la H_{\tig_1}k_z, H_{\tig_i}k_z\ra}{||H_{\tig_1}k_z||^2}$ is bounded, we may assume
$$
\lim_{z\to m}\frac{\la H_{\tig_1}k_z, H_{\tig_i}k_z\ra}{||H_{\tig_1}k_z||^2}=t^{(i)}_m.
$$
Therefore
$$
\lim_{z\to m}||H_{\sum_{i=1}^N t^{(i)}_m f_i}k_z||=0,
$$
and then $(\sum_{i=1}^N t^{(i)}_m f_i)|_{S_m}\in H^\infty|_{S_m}$, which contradicts the fact that $$\{[f_1]_m,...,[f_N]_m\}$$ is a basis.

(2) Notice that
$$
H_{\vec{F}} k_z\otimes H^*_{\vec{G}}Uk_{\bz}=H_{(I-A)\vec{F}} k_z\otimes H^*_{\vec{G}}Uk_{\bz}+H_{\vec{F}} k_z\otimes H^*_{A^T\vec{G}}Uk_{\bz}.
$$
Using a similar argument as above, we can get $(A^T\vec{G})|_{S_m}\in H^\infty|_{S_m}$ from \eqref{qc2}. Assume $\sum_{i=1}^n H_{f_i}T_{g_i}$ is compact. Since $k_z\to 0$ weakly as $|z|\to 1^-$, we have
$$
\lim_{z\to m}||\sum_{i=1}^n H_{f_i}T_{g_i}k_z||=0.
$$
Let $B=(b_{ij})$.
By Proposition \ref{prop}, we have
\begin{align*}
&\sum_{i=1}^n H_{f_i}T_{g_i}k_z=\sum_{i=1}^n (H_{f_i-\sum_{j=1}^N b_{i j}f_j}T_{g_i}+H_{\sum_{j=1}^N b_{i j}f_j}T_{g_i})k_z\\
=&\sum_{i=1}^n H_{f_i-\sum_{j=1}^N b_{i j}f_j}T_{g_i}k_z+\sum_{i=1}^N H_{f_i}T_{\sum_{j=1}^n b_{ji} g_i}k_z\\
=&\sum_{i=1}^n H_{f_i-\sum_{j=1}^N b_{i j}f_j}T_{g_i}k_z+\sum_{i=1}^N H_{f_i\sum_{j=1}^n b_{ji} g_i}k_z-\sum_{i=1}^N S_{f_i}H_{\sum_{j=1}^n b_{ji} g_i}k_z.
\end{align*}
Notice that $(f_i-\sum_{j=1}^N b_{i j}f_j)|_{S_m}\in H^\infty|_{S_m}$. By Lemma \ref{*lim}, we have
$$
\lim_{z\to m}||H_{f_i-\sum_{j=1}^N b_{i j}f_j}T_{g_i}k_z||=0.
$$
Also $(A^T\vec{G})|_{S_m}\in H^\infty|_{S_m}$ means that $$\lim_{z\to m}||H_{\sum_{j=1}^n b_{ji} g_i}k_z||=0.$$
Therefore, we have
$$
0=\lim_{z\to m}||\sum_{i=1}^N H_{f_i\sum_{j=1}^n b_{ji} g_i}k_z||=\lim_{z\to m}||H_{\vec{G}^T A \vec{F}}k_z||,
$$
and the conclusion follows from Lemma \ref{lim}.
\end{proof}

We also have a converse of Lemma \ref{qc}.
\begin{lem}\label{qcc}
Let $\vec{F}=(f_1,..,f_n)^T, \vec{G}=(g_1,...,g_n)^T$,  $\vec{H}=(h_1,...,h_n)^T$ and $\vec{F},\vec{G}, \vec{H}\in L^\infty$. Let $m\in M(H^\infty+C)$.
\begin{enumerate}
\item If there exists a scalar matrix $A_{n\times n}$ such that
$$(A\vec{F}-\vec{F})|_{S_m}\in H^\infty|_{S_m}, \, (A^*\vec{G})|_{S_m}\in H^\infty|_{S_m}.$$
Then \eqref{qc1} holds.
\item If there exists a scalar matrix $A_{n\times n}$ such that
$$(A\vec{F}-\vec{F})|_{S_m}\in H^\infty|_{S_m},\,(A^T\vec{H})|_{S_m}\in H^\infty|_{S_m},\, (\vec{H}^T A \vec{F})|_{S_m}\in H^\infty|_{S_m}$$
Then \eqref{qc2} holds and
\beq\label{qcc3}
\lim_{z\to m}||\sum_{i=1}^n T_{\bh_i}H^*_{f_i}Uk_{\bz} ||=0.
\eeq
\end{enumerate}
\end{lem}
\begin{proof}
(1) The conditions imply
$$\lim_{z\to m}||H_{(I-A)\vec{F}}k_z||=\lim_{z\to m}||H_{A^*\vec{G}}k_z||=0.$$
It then follows from \eqref{qcc1}.

(2) Use a similar argument we get \eqref{qc2}. We can check \eqref{qcc3} using the following identity
\begin{align*}
&\sum_{i=1}^n T_{\bh_i}H^*_{f_i}Uk_{\bz}=\sum_{i=1}^n (T_{\bh_i}H^*_{f_i-\sum_{j=1}^n a_{i j}f_j}+T_{\bh_i}H^*_{\sum_{j=1}^n a_{i j}f_j})Uk_{\bz}\\
=&\sum_{i=1}^n T_{\bh_i}H^*_{f_i-\sum_{j=1}^n a_{i j}f_j}Uk_{\bz}+\sum_{i=1}^n T_{\sum_{j=1}^n \ba_{ji} \bh_i}H^*_{f_i}Uk_{\bz}\\
=&\sum_{i=1}^n T_{\bh_i}H^*_{f_i-\sum_{j=1}^n a_{i j}f_j}Uk_{\bz}+\sum_{i=1}^n H^*_{ f_i\sum_{j=1}^n a_{ji} h_i}Uk_{\bz}-\sum_{i=1}^n H^*_{\sum_{j=1}^n a_{ji} h_i} S_{\barf_i}Uk_{\bz}\\
=&\sum_{i=1}^n T_{\bh_i}H^*_{f_i-\sum_{j=1}^n a_{i j}f_j}Uk_{\bz}-\sum_{i=1}^n H^*_{\sum_{j=1}^n a_{ji} h_i} S_{\barf_i}Uk_{\bz}+H^*_{\vec{H}^TA\vec{F}}Uk_{\bz}.
\end{align*}

\end{proof}

The following lemma will be used several times later.
\begin{lem}\label{bar}
Let $f\in L^\infty$, $m\in M(H^\infty+C)$. Let $\Gt$ be an inner function.
\begin{enumerate}
\item If $f|_{S_m}, \barf|_{S_m}\in H^\infty|_{S_m}$, then $f|_{S_m}$ is constant.
\item If $f|_{S_m}, (\Gt\barf)|_{S_m}\in H^\infty|_{S_m}$, then $f|_{S_m}\in (K_\Gt+\CC\Gt)|_{S_m}.$
\end{enumerate}
\end{lem}
\begin{proof}
\begin{enumerate}
\item Let $g_1, g_2\in H^\infty$ such that $f=g_1, \barf =g_2$ on $S_m$. Since $m$ is multiplicative on $H^\infty$, we have
\begin{align*}
&\int_{S_m}|f|^2 d\mu_m=\int_{S_m} g_1g_2 d\mu_m=m(g_1g_2)=m(g_1)m(g_2)\\
=&\int_{S_m}fd\mu_m \int_{S_m}\barf d\mu_m=|\int_{S_m}fd\mu_m|^2.
\end{align*}
Thus
$$
\int_{S_m}|f-\int_{S_m}fd\mu_m|^2d\mu_m=0,
$$
and then $f$ is the constant $\int_{S_m}fd\mu_m$ on $S_m$.
\item Let $g\in H^\infty$ such that $f=g$ on $S_m$. We can write $g=g_1+\Gt g_2$, where $g_1\in K_\Gt, g_2\in H^\infty$. Then $\Gt\barf=\Gt\bg=\Gt \bg_1+\bg_2$ on $S_m$. Since $\Gt\bg_1\in H^\infty$, we have $\bg_2|_{S_m}\in H^\infty|_{S_m}$. By part (1), $g_2|_{S_m}$ is constant and $g\in K_\Gt+\CC\Gt$.
\end{enumerate}
\end{proof}

Now we prove the main result in this section.
\begin{thm}\label{compact}
Let $f\in L^\infty$. Let $g=g_1+\ol{g_2}$, where $g_1, g_2 \in K_\Gt\cap H^\infty$ .Then $H^\Gt_f A^\Gt_g$ is compact if and only if for every $m\in M(H^\infty+C)$, one of the following holds
\begin{enumerate}
\item $\Gt|_{S_m}$ is constant.
\item $f|_{S_m}$ is constant.
\item
\begin{enumerate}
\item $\bg|_{S_m}, (\ol{f g})|_{S_m} \in H^\infty|_{S_m}.$
\item $f|_{S_m}\in (K_\Gt+\CC\Gt)|_{S_m}$.
\end{enumerate}
\item there exist $\Ga, \Gb\in \CC$ such that
\begin{enumerate}
\item $(\bar{\Gt}g_2-{\Ga} \bar{\Gt})|_{S_m}$ is constant.
\item $(\bg_1-\Gb\bar{\Gt})|_{S_m}=-\Ga.$
\end{enumerate}
\item there exists $\Ga\in\CC$ such that
\begin{enumerate}
\item $(\bar{\Gt}g_2-\Ga \bg_1)|_{S_m}$ is constant.
\item $(fg_1)|_{S_m}, (\bar{\Ga}f-\Gt f)|_{S_m}\in (K_\Gt+\CC\Gt)|_{S_m}.$
\end{enumerate}
\item there exist $\Ga, \Gb, C\in \CC$ such that
\begin{enumerate}
\item $(\bar{\Gt}-\Ga\bar{\Gt}g_2-\Gb \bg_1)|_{S_m}=C$.
\item $(f+\bar{\Ga} fg_1)|_{S_m}, (\bar{\Gb}f g_1-\Gt f)|_{S_m}\in(K_\Gt+\CC\Gt)|_{S_m}$.
\item $(\barf g_2-C\Gt\ol{fg_1})|_{S_m}\in H^\infty|_{S_m}$.
\end{enumerate}

\end{enumerate}
\end{thm}
\begin{proof}
By Lemmas \ref{eq1} and \ref{eq2}, $H^\Gt_f A^\Gt_g$ is compact if and only if
$$H^*_{\bar{\Gt}}H_{fg_1}+H^*_{\bar{\Gt}g_2}H_f+H^*_{\bg_1}H_{-\Gt f}$$
and
$$H_{\bar{\Gt}}T_{\Gt\ol{fg_1}}+H_{\bar{\Gt}g_2}T_{\Gt\barf}+H_{\bg_1}T_{-\barf}$$
are compact.
By Theorem \ref{CC}, the above conditions are equivalent to:
for every $m\in M(H^\infty+C)$,
\beq\label{c03}
\lim_{z\to m}||H_{\bar{\Gt}}k_z\otimes H_{fg_1}k_z+H_{\bar{\Gt}g_2}k_z\otimes H_f k_z+H_{\bg_1}k_z\otimes H_{-\Gt f}k_z||=0,
\eeq
\beq\label{c04}
\lim_{z\to m}||H_{\bar{\Gt}}k_z\otimes H^*_{\Gt\ol{fg_1}}U k_{\bz}+H_{\bar{\Gt}g_2}k_z\otimes H^*_{\Gt\barf}U k_{\bz}+H_{\bg_1}k_z\otimes H^*_{-\barf}U k_{\bz}||=0,
\eeq
and
\beq\label{c05}
\lim_{z\to m}||(T_{\bar{\Gt}f g_1}H^*_{\bar{\Gt}}+T_{\bar{\Gt}f}H^*_{\bar{\Gt}g_2}+T_{-f}H^*_{\bg_1} )Uk_{\bz}||=0.
\eeq

{\bf Necessity:}
Suppose $H^\Gt_f A^\Gt_g$ is compact. We consider the rank of $$\{[\bar{\Gt}]_m, [\bar{\Gt}g_2]_m, [\bg_1]_m\}.$$

Case I: $$rank\{[\bar{\Gt}]_m, [\bar{\Gt}g_2]_m, [\bg_1]_m\}=0.$$ In particular, $[\bar{\Gt}]_m=0$. which means $\bar{\Gt}|_{S_m}\in H^\infty|_{S_m}.$ By Lemma \ref{bar}, condition (1) holds.

Case II: $$rank\{[\bar{\Gt}]_m, [\bar{\Gt}g_2]_m, [\bg_1]_m\}=3.$$ Then $\{[\bar{\Gt}]_m, [\bar{\Gt}g_2]_m, [\bg_1]_m\}$ are linearly independent. By Lemma \ref{qc}, we have $f|_{S_m}\in H^\infty|_{S_m}, \barf|_{S_,}\in H^\infty|_{S_m}$. Thus condition (2) holds.

Case III: $$rank\{[\bar{\Gt}]_m, [\bar{\Gt}g_2]_m, [\bg_1]_m\}=1.$$ Notice that we only need to consider the case: $$\{[\bar{\Gt}]_m, [\bar{\Gt}g_2]_m, [\bg_1]_m\}=span\{[\bar{\Gt}]_m\},$$ which means there exist $\Ga, \Gb\in \CC$ such that $(\bar{\Gt}g_2-\Ga \bar{\Gt})|_{S_m}, (\bg_1-\Gb \bar{\Gt})|_{S_m}$ are constants. In fact, if $$\{[\bar{\Gt}]_m, [\bar{\Gt}g_2]_m, [\bg_1]_m\}=span\{[\bar{\Gt}g_2]_m\},$$ we have
$(\bar{\Gt}-\tilde{\Ga}\bar{\Gt}g_2)|_{S_m}, (\bg_1-\tilde{\Gb}\bar{\Gt}g_2)|_{S_m}$ are constants, for some $\tilde{\Ga}, \tilde{\Gb}\in\CC$. If $\tilde{\Ga}=0$, then $\Gt|_{S_m}$ is a constant, which is Case I. If not, we have $(\bar{\Gt}g_2-{1\over \tilde{\Ga}} \bar{\Gt})|_{S_m}, (\bg_1-{\tilde{\Gb}\over \tilde{\Ga}} \bar{\Gt})|_{S_m}$ are constants.

Take \beq\label{4}\vec{F}=(\bar{\Gt}, \bar{\Gt}g_2, \bg_1)^T, \vec{G}=(fg_1, f, -\Gt f)^T, \vec{H}=(\Gt\ol{fg_1}, \Gt\barf, -\barf)^T,\eeq
and $A=\begin{pmatrix}
1 & 0 & 0\\
\Ga & 0 & 0\\
\Gb& 0 &0
\end{pmatrix}$ in Lemma \ref{qc}, and use
$$
A^*\vec{G}=(fg_1+\bar{\Ga}f-\bar{\Gb}\Gt f,0 ,0)^T,\q A^T\vec{H}=(\Gt\ol{fg_1}+\Ga\Gt\barf-\Gb \barf, 0, 0)^T,
$$
we have
$u|_{S_m}, (\Gt\bu)|_{S_m}\in H^\infty|_{S_m}$, where $$u=f(g_1+\bar{\Ga} -\bar{\Gb} \Gt).$$
In addition, $(\vec{H}^TA\vec{F})|_{S_m}=\bu|_{S_m}\in H^\infty|_{S_m}$, and thus $u|_{S_m}$ is constant.
Let $(\bg_1-\Gb\bar{\Gt})|_{S_m}=C$ ($C$ is a constant). Then $$u|_{S_m}=(f(\bar{C}+\bar{\Ga}))|_{S_m}.$$
If $C+\Ga=0$, then we get condition (4).
If $C+\Ga\neq 0$, this means $f|_{S_m}\in H^\infty|_{S_m}$, which gives condition (2).

Case IV: $$rank\{[\bar{\Gt}]_m, [\bar{\Gt}g_2]_m, [\bg_1]_m\}=2.$$
Case IV(A): If $$\{[\bar{\Gt}]_m, [\bar{\Gt}g_2]_m, [\bg_1]_m\}=span\{[\bar{\Gt}g_2]_m, [\bg_1]_m\},$$
then there exist $\Ga, \Gb\in\CC$ such that $(\bar{\Gt}-\Ga \bar{\Gt}g_2-\Gb \bg_1)|_{S_m}=C$, and
$$\lim_{z\to m}||H_{\bar{\Gt}-\Ga \bar{\Gt}g_2-\Gb \bg_1}k_z||=0.$$

Take \beq\label{6}\vec{F}=(\bar{\Gt}g_2, \bg_1, \bar{\Gt})^T, \vec{G}=( f, -\Gt f, fg_1)^T, \vec{H}=( \Gt\barf, -\barf, \Gt\ol{fg_1})^T,\eeq
and $A=
\begin{pmatrix}
1 & 0 & 0\\
0 & 1 & 0\\
\Ga & \Gb & 0
\end{pmatrix}$
in Lemma \ref{qc}, we have
$$(A^*\vec{G})|_{S_m}=(f+\bar{\Ga}fg_1, -\Gt f+\bar{\Gb}fg_1, 0)^T|_{S_m}\in H^\infty|_{S_m}, $$
and
$$(A^T\vec{H})|_{S_m}=(\Gt \barf+\Ga\Gt\ol{fg_1}, -\barf+\Gb\Gt\ol{fg_1}, 0)^T|_{S_m}\in H^\infty|_{S_m}.$$
Thus
$u|_{S_m}, (\Gt\bu)|_{S_m}, v|_{S_m}, (\Gt\bv)|_{S_m}\in H^\infty|_{S_m}$, where
$$
u=f+\bar{\Ga} fg_1,\q v=f\bar{\Gb} g_1-\Gt f.
$$
On the other hand, Lemma \ref{bar} implies that $u|_{S_m}, v|_{S_m}\in (K_\Gt+\CC\Gt)|_{S_m}$.
Using Lemma \ref{qc} again, we have
\begin{align*}
&(\vec{H}^TA\vec{F})|_{S_m}=(g_2\bu+\Gt\bg_1\bv)|_{S_m}\\
=&(\barf g_2(1+\Ga\bg_1)+\Gt\ol{fg_1}(\Gb\bg_1-\bar{\Gt}))|_{S_m}\\
=&(\barf g_2(1+\Ga\bg_1)-\Gt\ol{fg_1}(\Ga\bar{\Gt}g_2+C))|_{S_m}\\
=&(\barf(g_2-C\Gt\bg_1))|_{S_m}\in H^\infty|_{S_m},
\end{align*}
which gives condition (6).

If $$\{[\bar{\Gt}]_m, [\bar{\Gt}g_2]_m, [\bg_1]_m\}=span\{[\bar{\Gt}]_m, [\bar{\Gt}g_2]_m\},$$
or $$\{[\bar{\Gt}]_m, [\bar{\Gt}g_2]_m, [\bg_1]_m\}=span\{[\bar{\Gt}]_m, [\bg_1]_m\},$$
then there exist $t_1, t_2, t_3\in\CC$ such that
$(t_1\bar{\Gt}+t_2\bar{\Gt}g_2+t_3\bg_1)|_{S_m}$ is constant, where $|t_2|^2+|t_3|^2\neq 0.$
If $t_1\neq 0$, then we have Case IV(A). If $t_1=0$, we have the following two cases:
$g_1|_{S_m}$ is constant, or $(\bar{\Gt}g_2-\Ga \bg_1)|_{S_m}$ is constant for some $\Ga\in\CC$.

Case IV(B): $g_1|_{S_m}=C\in\CC$.
Then $$\lim_{z\to m}||H_{\bg_1}k_z||=0.$$
Since we also assume $[\Gt]_m, [\bar{\Gt}g_2]_m$ are linearly independent, by Lemma \ref{qc} and \eqref{c03}, \eqref{c04}, we have
$(fg_1)|_{S_m}, f|_{S_m}, (\Gt\ol{fg_1})|_{S_m}, (\Gt\barf)|_{S_m}$ and $$(\bar{\Gt}\cdot\Gt \ol{fg_1}+\bar{\Gt}g_2\cdot \Gt\barf)|_{S_m}=(\barf(\bg_1+g_2))|_{S_m}$$ are in $H^\infty|_{S_m}.$ Thus $f|_{S_m}\in (K_\Gt+\CC\Gt)|_{S_m}$ and $(\barf(\bar{C}+g_2))|_{S_m}=(\ol{f g})|_{S_m}\in H^\infty|_{S_m}$, and condition (3) holds.

Case IV(C): $(\bar{\Gt}g_2-\Ga \bg_1)|_{S_m}$ is constant for some $\Ga\in\CC$.
We use Lemma \ref{qc} for
\beq\label{5}\vec{F}=(\bar{\Gt}, \bg_1, \bar{\Gt}g_2)^T, \vec{G}=(fg_1,-\Gt f, f)^T, \vec{H}=( \Gt\ol{fg_1}, -\barf, \Gt\barf)^T,\eeq
and $A=
\begin{pmatrix}
1 & 0 & 0\\
0 & 1 & 0\\
0 & \Ga & 0
\end{pmatrix}$. Then
$$(A^*\vec{G})|_{S_m}=(fg_1, -\Gt f+\bar{\Ga}f, 0)^T|_{S_m}\in H^\infty|_{S_m},$$
$$(A^T\vec{H})|_{S_m}=(\Gt\ol{fg_1}, -\barf+\Ga\Gt\barf, 0)^T|_{S_m}\in H^\infty|_{S_m},$$
and $$(\vec{H}^TA\vec{F})|_{S_m}=\Ga(\Gt\ol{fg_1})|_{S_m}\in H^\infty|_{S_m}.$$
By Lemma \ref{bar}, we have condition (5).

{\bf Sufficiency:} For every $m\in M(H^\infty+C)$, we need to check \eqref{c03}, \eqref{c04}, \eqref{c05} under each of the conditions (1)--(6).

If condition (1) holds, then$$\lim_{z\to m}||H_{\bar{\Gt}}k_z||=0.$$ By Lemma \ref{*lim} we have
$$
\lim_{z\to m}||H_{\bar{\Gt}g_2}k_z||=\lim_{z\to m}||H_{\bar{\Gt}}T_{g_2}k_z||=0,
$$
$$
\lim_{z\to m}||H_{\bg_1}k_z||=\lim_{z\to m}||H_{\bar{\Gt}}T_{\Gt \bg_1}k_z||=0.
$$
Thus \eqref{c05} follows from Lemma \ref{lim}.

If condition (2) holds, then
$$
\lim_{z\to m}||H_{f}k_z||=\lim_{z\to m}||H^*_{-\barf}Uk_{\bz}||=0,
$$
Notice that
$$
H_{fg_1}=H_fT_{g_1}, H_{-\Gt f}=H_fT_{-\Gt}, H^*_{\Gt \ol{fg_1}}=H^*_{\barf}S_{\Gt\bg_1}, H^*_{\Gt\barf}=H^*_{\barf}S_{\Gt},
$$
Lemma \ref{*lim} implies
\beq\label{(2)}
\lim_{z\to m}||H_{fg_1}k_z||=\lim_{z\to m}||H_{-\Gt f}k_z||=\lim_{z\to m}||H^*_{\Gt \ol{fg_1}}Uk_{\bz}||=\lim_{z\to m}||H^*_{\Gt\barf}Uk_{\bz}||=0.
\eeq
Thus \eqref{c03}, \eqref{c04} hold. To check \eqref{c05}, we use
\begin{align*}
&T_{\bar{\Gt}f g_1}H^*_{\bar{\Gt}}+T_{\bar{\Gt}f}H^*_{\bar{\Gt}g_2}+T_{-f}H^*_{\bg_1}\\
=&H^*_{\ol{fg_1}}-H^*_{\Gt\ol{f g_1}}S_{\Gt}+H^*_{\barf g_2}-H^*_{\Gt\barf}S_{\Gt\bg_2}+H^*_{-\ol{fg_1}}-H^*_{-\barf}S_{g_1}\\
=&H^*_{\barf g_2}-H^*_{\Gt\ol{f g_1}}S_{\bar{\Gt}}-H^*_{\Gt\barf}S_{\bar{\Gt}g_2}-H^*_{-\barf}S_{\bg_1}.
\end{align*}
Then \eqref{c05} follows from \eqref{(2)} and Lemma \ref{*lim}.

If condition (3) holds, then
$$
\lim_{z\to m}||H_{g_1}k_z||=\lim_{z\to m}||H_{fg_1}k_z||=\lim_{z\to m}||H^*_{\Gt\barf}Uk_{\bz}||=\lim_{z\to m}||H^*_{\Gt\ol{fg_1}}Uk_{\bz}||=0,
$$
which implies \eqref{c03}, \eqref{c04}. On the other hand, by Lemma \ref{*lim} and
$$
T_{\bar{\Gt}f g_1}H^*_{\bar{\Gt}}+T_{\bar{\Gt}f}H^*_{\bar{\Gt}g_2}+T_{-f}H^*_{\bg_1}
=H^*_{\ol{fg_1}}-H^*_{\Gt\ol{f g_1}}S_{{\Gt}}+H^*_{\barf g_2}-H^*_{\Gt\barf}S_{\Gt\bg_2}+T_{-f}H^*_{\bg_1},
$$
we get \eqref{c05}.

If condition (4) holds, we use \eqref{4} in Lemma \ref{qcc}.

If condition (5) holds, we use \eqref{5} in Lemma \ref{qcc}.

If condition (6) holds, we use \eqref{6} in Lemma \ref{qcc}.
\end{proof}

\bibliography{references}

\begin{bibdiv}
\begin{biblist}

\bib{ax78}{article}{
      author={Axler, S.},
      author={Chang, S.-Y.},
      author={Sarason, D.},
       title={Product of {Toeplitz} operators},
        date={1978},
     journal={Inter. Equ. Oper. Th.},
      volume={1},
       pages={285\ndash 309},
}

\bib{bcfmd}{article}{
      author={Baranov, A.},
      author={Chalendar, I.},
      author={Fricain, E.},
      author={Mashreghi, J.},
      author={Timotin, D.},
       title={Bounded symbols and reproducing kernel thesis for truncated
  {T}oeplitz operators},
        date={2010},
     journal={J. Funct. Anal.},
      volume={259},
      number={10},
       pages={2673\ndash 2701},
}

\bib{ba82}{article}{
      author={Barr\'ia, J.},
      author={Halmos, P.},
       title={Asymptotic {Toeplitz} operators},
        date={1982},
     journal={Trans. Amer. Math. Soc},
      volume={273},
       pages={621\ndash 630},
}

\bib{car62}{article}{
      author={Carleson, L.},
       title={Interpolations by bounded analytic functions and the corona
  problem},
        date={1962},
     journal={Ann. of Math.},
      volume={76},
       pages={547\ndash 559},
}

\bib{chu15}{article}{
      author={Chu, C.},
       title={Compact product of {Hankel} and {T}oeplitz operators},
        date={2015},
     journal={Indiana Univ. Math. J.},
      volume={64},
       pages={973\ndash 982},
}

\bib{dou72}{book}{
      author={Douglas, R.~G.},
       title={{Banach} algebra techniques in operator theory},
   publisher={Academic Press},
     address={New York},
        date={1972},
}

\bib{gar81}{book}{
      author={Garnett, John~B.},
       title={Bounded analytic functions},
   publisher={Academic Press},
     address={New York},
        date={1981},
}

\bib{gor99}{article}{
      author={Gorkin, P.},
      author={Zheng, D.},
       title={Essentially commuting {Toeplitz} operators},
        date={1999},
     journal={Pacific J. Math.},
      volume={190},
       pages={87\ndash 109},
}

\bib{guz98}{article}{
      author={Gu, C.},
      author={Zheng, D.},
       title={Products of block {T}oeplitz operators},
        date={1998},
     journal={Pacific J. Math.},
      volume={185},
      number={1},
       pages={115\ndash 148},
}

\bib{guo03}{article}{
      author={Guo, K.},
      author={Zheng, D.},
       title={Essentially commuting {Hankel} and {Toeplitz} operators},
        date={2003},
     journal={J. Funct. Anal.},
      volume={201},
       pages={121\ndash 147},
}

\bib{guo05}{article}{
      author={Guo, K.},
      author={Zheng, D.},
       title={The distribution function inequality for a finite sum of finite
  products of {Toeplitz} operators},
        date={2005},
     journal={J. Funct. Anal.},
      volume={218},
       pages={1\ndash 53},
}

\bib{mazheng2}{article}{
      author={Ma, P.},
      author={Yan, F.},
      author={Zheng, D.},
       title={Zero, finite rank, and compact big truncated {H}ankel operators
  on model spaces},
        date={2018},
     journal={Proc. Amer. Math. Soc.},
      volume={146},
      number={12},
       pages={5235\ndash 5242},
}

\bib{mazheng}{article}{
      author={Ma, P.},
      author={Zheng, D.},
       title={Compact truncated {T}oeplitz operators},
        date={2016},
     journal={J. Funct. Anal.},
      volume={270},
      number={11},
       pages={4256\ndash 279},
}

\bib{pel02}{book}{
      author={Peller, V.V.},
       title={Hankel operators and their applications},
   publisher={Springer},
     address={New York},
        date={2002},
}

\bib{sar67}{article}{
      author={Sarason, D.},
       title={Generalized interpolation in {$H^\infty$}},
        date={1967},
     journal={Trans.\ Amer.\ Math.\ Soc.},
      volume={127},
       pages={179\ndash 203},
}

\bib{sar07}{article}{
      author={Sarason, D.},
       title={Algebraic properties of truncated {T}oeplitz operators},
        date={2007},
     journal={Oper. Matrices},
      volume={1},
      number={4},
       pages={491\ndash 526},
}

\bib{vol82}{article}{
      author={Volberg, A.},
       title={Two remarks concerning the theorem of {S. Axler}, {S.-Y. A.
  Chang}, and {D. Sarason}},
        date={1982},
     journal={J. Operator Theory},
      volume={8},
       pages={209\ndash 218},
}

\bib{zh96}{article}{
      author={Zheng, D.},
       title={The distribution function inequality and products of {Toeplitz}
  operators and {Hankel} operators},
        date={1996},
     journal={J. Funct. Anal.},
      volume={138},
      number={2},
       pages={477\ndash 501},
}

\bib{zhu}{book}{
      author={Zhu, K.},
       title={Operator theory in function spaces},
      series={Monographs and textbooks in pure and applied mathematics},
   publisher={Marcel Dekker, Inc.},
     address={New York},
        date={1990},
      volume={139},
}

\bib{zhu2}{book}{
      author={Zhu, K.},
       title={Operator theory in function spaces, second edition},
      series={Mathematical Surveys and Monographs},
   publisher={American Mathematical Society},
     address={Providence, RI},
        date={2007},
      volume={138},
}

\end{biblist}
\end{bibdiv}
\end{document}